\documentclass[3p,final]{elsarticle}
\usepackage{amsmath,amssymb,amsthm}
\usepackage[utf8]{inputenc}
\usepackage{amsfonts}
\usepackage{color}
\usepackage{booktabs}

\def\mudz{\mu\partial_z}

\def\R{\mathcal{R}}
\def\R{\mathcal{R}}

\def\S{\mathcal{S}}

\def\P{\mathcal{P}}

\def\RR{\mathbb{R}}

\def\NN{\mathbb{N}}
\def\WW{\mathbb{W}}

\def\slabwidth{Z}
\def\DLmu{\mathcal{D}}

\newtheorem{theorem}{Theorem}[section]
\newtheorem{lemma}[theorem]{Lemma}
\newtheorem{problem}[theorem]{Problem}

\newtheorem{remark}[theorem]{Remark}

\newcommand\matthias[1]{{\color{black}#1}}

\begin{document}
	
	\begin{frontmatter}


	\journal{Computers and Mathematics with Applications}

	\title{On a convergent DSA preconditioned source iteration for a DGFEM method for radiative transfer}
	\author{Olena Palii}
	\ead{o.palii@utwente.nl}
%

	\author{Matthias Schlottbom\corref{cor1}}
	\ead{m.schlottbom@utwente.nl}
	\address{Department of Applied Mathematics, University of Twente, P.O. Box 217, 7500 AE Enschede, The Netherlands}
	\cortext[cor1]{Corresponding author}

	\begin{abstract}
	We consider the numerical approximation of the radiative transfer equation using discontinuous angular and continuous spatial approximations for the even parts of the solution. The even-parity equations are solved using a diffusion synthetic accelerated source iteration. We provide a convergence analysis for the infinite-dimensional iteration as well as for its discretized counterpart. The diffusion correction is computed by a subspace correction, which leads to a convergence behavior that is robust with respect to the discretization. The proven theoretical contraction rate deteriorates for scattering dominated problems. We show numerically that the preconditioned iteration is in practice robust in the diffusion limit. Moreover, computations for the lattice problem indicate that the presented discretization does not suffer from the ray effect. The theoretical methodology is presented for plane-parallel geometries with isotropic scattering, but the approach and proofs generalize to multi-dimensional problems and more general scattering operators verbatim.
	\end{abstract}


	\end{frontmatter}
	

\section{Introduction}
We consider the numerical solution of the radiative transfer equation in plane parallel geometry
\begin{alignat}{2}
\mudz \phi(z,\mu) + \sigma_t(z) \phi(z,\mu) &= \frac{\sigma_s(z)}{2}\int_{-1}^1 \phi(z,\mu')d\mu' + q(z,\mu), \label{eq:rte1} 
\end{alignat}
where $0<z<\slabwidth$ and $-1<\mu<1$, and $\slabwidth$ denotes the thickness of the slab and $\mu$ is the cosine of the polar angle of a unit vector. The function $\phi(z,\mu)$ models the equilibrium distribution of some quantity, like neutrons or photons \cite{Chandrasekhar60,CaseZweifel67}. The basic physical principles embodied in \eqref{eq:rte1} are transport, which is modeled by the differential operator $\mudz$, attenuation with rate $\sigma_t$ and scattering with $\sigma_s$. Internal sources are described by the function $q$. In this work we will close the radiative transfer equation using inflow boundary conditions
\begin{align}\label{eq:rte2}
	\phi(0,\mu)=g^0(\mu)\quad \mu>0,\quad\text{and}\quad \phi(\slabwidth,\mu)=g^\slabwidth(\mu)\quad \mu<0.
\end{align}
Such transport problems arise when the full three-dimensional model posed on $\RR^2\times(0,\slabwidth)\times\S^2$ obeys certain symmetries \cite{CaseZweifel67}. It has been studied in many instance due to simpler structure compared to three dimensional problems without symmetries; while the methodology presented here directly carries over to the general case, see also the numerical examples presented below.

Classical deterministic discretization strategies are based on a semidiscretization in $\mu$. 
One class of such strategies are the $P_N$-approximations, which are spectral methods based on truncated spherical harmonics expansions \cite{Vladimirov61,LewisMiller84}, and we refer to \cite{Pitkaranta1977,ManResSta00,EggerSchlottbom12} for variational discretization strategies using this approximation.
The major advantage of $P_N$-approximations is that the scattering operator becomes diagonal. In addition, the matrix representation of the transport operator $\mudz$ is sparse. The main drawbacks of the $P_N$-method are that the variational incorporation of the inflow boundary condition introduces a dense coupling of the spherical harmonics expansion coefficients making standard $P_N$-approximations quite expensive to solve. We note, however, that a modified variational formulation of the $P_N$-equations has recently been derived that leads to sparse matrices \cite{EggerSchlottbom18}.
In any case, the success of spectral approximation techniques depends on the smoothness of the solution.
In general, the solution $\phi$ is not smooth for $\mu=0$, which is related to the inflow boundary conditions \eqref{eq:rte2}. Hence, the $P_N$-approximations will in general not converge spectrally.

A second class of semidiscretizations are discrete ordinates methods that use a quadrature rule for the discretization of the $\mu$-variable \cite{CaseZweifel67}, with analysis provided in \cite{JohnsonPitkaranta83,Asadzadeh86}.
Such methods are closely related to discontinuous Galerkin (DG) methods, see, e.g., \cite{ReedHill73,WareingMcGheeMorelPautz2001,RagusaGuermondKanschat12,KanschatRagusa14,KophaziLathouwers2015}.
While allowing for local angular resolution, the main obstruction in the use of these methods is that the scattering operator leads to dense matrices, and a direct inversion of the resulting system is not possible in realistic applications. To overcome this issue, iterative solution techniques have been proposed. 
An often used iterative technique is Richardson iteration, i.e., the source iteration \cite{MarchukLebedev86,AdamsLarsen02}, but other Krylov space methods exist \cite{WarsaWareingMorel2004}. 
The key idea in the source iterations is to decouple scattering and transport, and to exploit that the transport part can be inverted efficiently.
If $\sigma_s/\sigma_t\approx 1$, the convergence of these iterative methods is slow, and several preconditioning techniques have been proposed \cite{KanschatRagusa14,MarchukLebedev86,AdamsLarsen02,WarsaWareingMorel2004}.
Among the most used and simple preconditioners is the diffusion synthetic acceleration method (DSA), in which a diffusion problem is solved in every iteration. This is well motivated by asymptotic analysis \cite{LarsenKeller,EggerSchlottbom2014}.
While Fourier analysis can be applied to special situations \cite{MarchukLebedev86,AdamsLarsen02}, the convergence analysis is mainly open for the general case, i.e., for jumping coefficients or non-periodic boundary conditions. 
A further complication in using the DSA preconditioner is that the resulting iterative scheme might diverge if the discretization of the diffusion problem and the discrete ordinates system is not consistent \cite{AdamsLarsen02}.

The contribution of this paper is to develop discretizations that allow for local resolution of the non-smoothness of the solution, and which lead to discrete problems that can be solved efficiently by diffusion synthetic accelerated source iterations. 
Our approach builds upon an even-parity formulation of the radiative transfer equations derived from the mixed variational framework given in \cite{EggerSchlottbom12}, where $P_N$-approximations have been treated in detail.
Using the framework of $KP$-methods \cite{MarchukLebedev86}, we show that an infinite dimensional DSA preconditioned source iteration converges already for the continuous problem.
We present conforming $hp$-type approximations spaces for the semidiscretization in $\mu$, and prove quasi-best approximation properties.
In order to solve resulting linear systems, we employ a DSA preconditioned Richardson iteration, which is just the infinite dimensional iteration projected to the approximation spaces. In particular the finite dimensional iteration is guaranteed to converge for any discretization. Moreover, the inversion of the transport problem can be parallelized straightforward.
In numerical experiments, even when employing low-order approximations, we observe that the developed method does not suffer from the ray effect, which is typically observed for discrete ordinates methods \cite{LewisMiller84,Brunner05}.

The outline of the paper is as follows:
In Section~\ref{sec:prelim} we introduce basic notation and recall the relevant function spaces.
In Section~\ref{sec:weak} we introduce the even-parity equation for \eqref{eq:rte1}, show its well-posedness, and formulate an infinite dimensional DSA preconditioned source iteration, for which we show convergence.
The approximation spaces are described in Section~\ref{sec:galerkin} and well-posedness of the Galerkin problems as well as quasi-best approximation results are presented.
In Section~\ref{sec:sn} we discuss the efficient iterative solution of the resulting linear systems and provide a convergence proof for the discrete DSA scheme.
Section~\ref{sec:numerics} presents supporting numerical examples for slab geometry and for multi-dimensional problems that show the good approximation properties of the proposed methods as well as fast convergence of the iterative solver in multi-dimensional problems and in the diffusion limit. The paper ends with some conclusions in Section~\ref{sec:conclusion}.

\section{Function spaces and further preliminaries}\label{sec:prelim}
Following \cite{Agoshkov98} we denote by $L^2(\DLmu)$ with $\DLmu=(0,\slabwidth)\times(-1,1)$ the usual Hilbert space of square integrable functions with inner product
\begin{align*}
	(\phi,\psi)=\int_{-1}^1 \int_0^\slabwidth \phi(z,\mu) \psi(z,\mu) dz d\mu
\end{align*}
and induced norm $\|\phi\|_{L^2(\DLmu)}=(\phi,\phi)^{\frac{1}{2}}$.
Furthermore, we define the Hilbert space
\begin{align*}
	H^1_2(\DLmu) = \{\phi\in L^2(\DLmu): \mudz\phi\in L^2(\DLmu)\}
\end{align*}
of functions with square integrable weak derivatives with respect to the weighted differential operator $\mudz$ endowed with the corresponding graph norm.  

In order to deal with boundary data, let us introduce the Hilbert space 
$L^2_-$ that consists of measurable functions for which
\begin{align*}
	\|\psi\|_{L^2_-}^2=\int_0^1  |\psi(0,\mu)|^2 |\mu|d\mu + \int_{-1}^0 |\psi(\slabwidth,\mu)|^2|\mu| d\mu
\end{align*}
is finite, and we denote by $\langle\psi,\phi\rangle_{L^2_-}$ the corresponding inner product on $L^2_-$.
Similarly, $L^2_+$ denotes the space of outflow data.
We have the following trace lemma \cite{Agoshkov98}.
\begin{lemma}\label{lem:trace}
	If $\phi\in H^1_2(\DLmu)$, then there exist traces $\phi_{\mid\Gamma_-}\in L^2_-$ and $\phi_{\mid\Gamma_+}\in L^2_+$ and
	\begin{align*}
		\|\phi_{\mid\Gamma_\pm}\|_{L^2_{\pm}}\leq \frac{C}{\sqrt{1-e^{-\slabwidth}}} \|\phi\|_{H^1_2(\DLmu)}
	\end{align*}
	with a constant $C>0$ independent of $\phi$ and $\slabwidth$.
\end{lemma}
As a consequence of the trace lemma and the density of smooth functions in $H^1_2(\DLmu)$ the following integration-by-parts formula is true
\begin{align}\label{eq:partial_integration}
		(\mudz \phi,\psi) = -(\phi,\mudz \psi) + \langle\phi,\psi\rangle_{L^2_+}-\langle\phi,\psi\rangle_{L^2_-}.
\end{align}
Throughout the manuscript we make the following basic assumption:
\begin{itemize}
	\item[(A1)] $\sigma_s,\sigma_t\in L^\infty(0,\slabwidth)$ are non-negative and $\sigma_a=\sigma_t-\sigma_s\geq \gamma>0$.
\end{itemize}
Assumption (A1) means that we consider absorbing media, which makes \eqref{eq:rte1}-\eqref{eq:rte2} well-posed \cite{Agoshkov98}.
\begin{lemma}\label{lem:existence}
	Let assumption (A1) hold, and let $g_-\in L^2_-$ and $q\in L^2(\DLmu)$, then \eqref{eq:rte1}-\eqref{eq:rte2} has a unique solution $\phi\in H^1_2(\DLmu)$ that satisfies the a-priori bound
	\begin{align*}
		\|\phi\|_{H^1_2{(\DLmu)}} \leq C (\|g_-\|_{L^2_-}+\|q\|_{L^2(\DLmu)}).
	\end{align*}
\end{lemma}
Assumption (A1) is not required to prove well-posedness for bounded geometries \cite{EggerSchlottbom2014Lp}, or for slab problems with constant coefficients \cite[Thm. 2.25]{Blake16}.

Let $\P:L^2(\DLmu)\to L^2(\DLmu)$ denote the $L^2$-projection onto constants in $\mu$, i.e.,
\begin{align*}
	(\P\psi)(z,\mu) = \frac{1}{2}\int_{-1}^1 \psi(z,\mu') d\mu'.
\end{align*}
Since $\sigma_t\in L^\infty(0,\slabwidth)$ is strictly positive, 
we can define the norms on $L^2(\DLmu)$ as follows
\begin{align}
	\|\psi\|_{\sigma_t}^2 = (\sigma_t \psi,\psi)\quad\text{and}\quad \|\psi\|_{\frac{1}{\sigma_t}}^2 = (\frac{1}{\sigma_t} \psi,\psi).
\end{align}
\subsection{Even-odd splitting}
The even $\phi^+$ and odd $\phi^-$ parts of a function $\phi\in L^2(\DLmu)$ are defined as
\begin{align*}
	\phi^+(z,\mu)=\frac{1}{2}( \phi(z,\mu)+\phi(z,-\mu)),\qquad 	\phi^-(z,\mu)=\frac{1}{2}( \phi(z,\mu)-\phi(z,-\mu)).
\end{align*}
Even-odd decompositions are frequently used in transport theory \cite{Pitkaranta1977,Vladimirov61}.
Since, as functions of $\mu$, even and odd function are orthogonal in $L^2(-1,1)$, we can decompose $L^2(\DLmu)$ into orthogonal subspaces containing even and odd functions, respectively,
\begin{align*}
	L^2(\DLmu) = L^2(\DLmu)^+ \oplus L^2(\DLmu)^-.
\end{align*}
Similarly, we will write $H^1_2(\DLmu)^\pm = H^1_2(\DLmu)\cap L^2(\DLmu)^\pm$.
As in \cite{EggerSchlottbom12}, we observe that $\mudz\phi^\pm \in L^2(\DLmu)^{\mp}$ for any $\phi\in H^1_2(\DLmu)$, and $\P\phi^\pm\in L^2(\DLmu)^+$ for $\phi\in L^2(\DLmu)$.
It turns out that the natural space for our formulation is
\begin{align*}
	\WW = H^1_2(\DLmu)^+ \oplus L^2(\DLmu)^-.
\end{align*}
%
\section{Weak formulation of the slab problem}\label{sec:weak}
\subsection{Derivation}
We follow the steps presented in \cite{EggerSchlottbom12} for multi-dimensional problems. The key idea is to rewrite the slab problem into a weak formulation for the even and odd parts of the solution. Multiplication of \eqref{eq:rte1} with a test function $\psi\in\WW^+$ and using orthogonality of even and odd functions gives that
\begin{align*}
	(\mudz\phi^-,\psi^+)+((\sigma_t-\sigma_s\P)\phi^+,\psi^+)=(q^+,\psi^+).
\end{align*}
Integration-by-parts \eqref{eq:partial_integration} applied to the first term on the left-hand side yields that
\begin{align*}
	(\mudz \phi^-,\psi^+) = -(\phi^-,\mudz \psi^+) + 
	\langle\phi^-,\psi^+\rangle_{L^2_+}-\langle\phi^-,\psi^+\rangle_{L^2_-}.
\end{align*}
Due to symmetries, we have that $\langle\phi^-,\psi^+\rangle_{L^2_+}=-\langle\phi^-,\psi^+\rangle_{L^2_-}$.
Using \eqref{eq:rte2}, we have that $\phi^-=\phi-\phi^+=g-\phi^+$ on the inflow boundary, which leads to
\begin{align*}
	(\mudz \phi^-,\psi^+) = -(\phi^-,\mudz \psi^+) + 
	2\langle\phi^+,\psi^+\rangle_{L^2_-}-2\langle g,\psi^+\rangle_{L^2_-}.
\end{align*}
Thus, for any $\psi^+\in\WW^+$, it holds that
\begin{align}\label{eq:weak1}
 	2\langle\phi^+,\psi^+\rangle_{L^2_-}-(\phi^-,\mudz \psi^+) + ((\sigma_t-\sigma_s\P) \phi^+,\psi^+)=(q^+,\psi^+)+2\langle g_-,\psi^+\rangle_{L^2_-}.
\end{align}
Testing \eqref{eq:rte1} with an odd test function $\psi^-\in \WW^-$, we obtain that
\begin{align*}
	(\mudz \phi^+,\psi^-)+(\sigma_t \phi^-,\psi^-)=(q^-,\psi^-),
\end{align*}
which implies that $\phi^-= \frac{1}{\sigma_t}(q^- - \mudz\phi^+)\in \WW^-$. Using this expression for $\phi^-$ in \eqref{eq:weak1}, we deduce that $u=\phi^+$ is a solution to the following problem; cf. \cite{EggerSchlottbom12}.
\begin{problem}\label{prob:cont}
 Let $q\in L^2(\DLmu)$ and $g\in L^2_-$. Find $u \in \WW^+$ such that 
\begin{align}\label{eq:even-parity}
	a(u,v)=\ell(v)\quad\text{for all } v\in\WW^+.
\end{align}
where the bilinear form $a:\WW^+\times\WW^+\to\RR$ is given by
\begin{align}\label{eq:def_a}
	a(u,v)&=2\langle u,v\rangle_{L^2_-}+(\frac{1}{\sigma_t}\mudz u,\mudz v)+((\sigma_t-\sigma_s\P) u,v),
\end{align}
and the linear form $\ell:\WW^+\to\RR$ is defined as
 \begin{align}\label{eq:def_RHS}
	\ell(v)=(q^+,v) + 2\langle g_-,\psi^+\rangle_{L^2_-}+(q^-,\frac{1}{\sigma_t}\mudz v).
 \end{align}
\end{problem}

\subsection{Well-posedness}
We endow $\WW^+$ with the norm induced by the bilinear form $a$ defined \eqref{eq:def_a}, i.e.,
\begin{align}\label{eq:energy_norm}
	\|u\|_\WW=\|u\|_a=a(u,u)^{\frac{1}{2}} \quad\text{for } u\in\WW^+.
\end{align}
Using the Cauchy-Schwarz inequality, we obtain the following result.
\begin{lemma}
	The linear form $\ell:\WW\to\RR$ defined in \eqref{eq:def_RHS} is bounded, i.e., for all $v\in\WW^+$ it holds
	\begin{align*}
		\ell(v)\leq (\|q^+\|_{\frac{1}{\sigma_a}}^2+\|q^-\|_{\frac{1}{\sigma_t}}^2+2\|g_-\|_{L^2_-}^2)^{\frac{1}{2}} \|v\|_a.
	\end{align*}
\end{lemma}
Since the space $\WW^+$ endowed with the inner product induced by $a$ is a Hilbert space, the unique solvability of Problem~\ref{prob:cont} is a direct consequence of the Riesz representation theorem.
\begin{theorem}\label{thm:solve_problem}
	Let Assumption (A1) hold true. Then Problem~\ref{prob:cont} has a unique solution $u\in\WW^+$. Moreover, we have the bound
	\begin{align*}
		\|u\|_{a} \leq (\|q^+\|_{\frac{1}{\sigma_a}}^2+\|q^-\|_{\frac{1}{\sigma_t}}^2+2\|g_-\|_{L^2_-}^2)^{\frac{1}{2}}.
	\end{align*}
\end{theorem}
\begin{remark}
	Setting $\phi^+=u$ and $\phi^-=\frac{1}{\sigma_t}(q^--\mudz u)$, one can show that $\mudz\phi^-\in L^2(\DLmu)$, i.e., $\phi=\phi^++\phi^-\in H^1_2(\DLmu)$ satisfies \eqref{eq:rte1} in $L^2(\DLmu)$, cf. \cite{EggerSchlottbom12}. Using the trace lemma~\ref{lem:trace} and partial-integration \eqref{eq:partial_integration}, we further can show that $\phi$ satisfies the boundary conditions \eqref{eq:rte2} in the sense of traces. Hence, Theorem~\ref{thm:solve_problem}, independently, leads to a well-posedness result as Lemma~\ref{lem:existence} for \eqref{eq:rte1}-\eqref{eq:rte2}.
\end{remark}

\subsection{DSA preconditioned source iteration in second order form}\label{sec:dsa_source}
As a preparation for the numerical solution of the discrete systems that will be described below, let us discuss an iterative scheme in infinite dimensions for solving the radiative transfer equation. The basic idea is a standard one and consists of decoupling scattering and transport in order to compute successive approximation, viz., the source iteration \cite{MarchukLebedev86,AdamsLarsen02}.
Next to the basic iteration, we describe a preconditioner which resembles diffusion synthetic acceleration (DSA) schemes using the notation of \cite{AdamsLarsen02} or a $KP_1$ scheme using the terminology of \cite{MarchukLebedev86}.

In order to formulate the method, we introduce the following bilinear forms
\begin{align*}
	k(u,v)&=(\sigma_s \P u,v)\text{ and } b(u,v)=a(u,v)+k(u,v)\quad\text{for }u,v,\in\WW^+,
\end{align*}
and denote the induced semi-norm and norm by $\|u\|_k$ and $\|u\|_b$, respectively.

The iteration scheme is defined as follows:
For $u^k\in\WW^+$ given, compute $u^{k+\frac{1}{2}}\in\WW^+$ as the unique solution to
\begin{align}\label{eq:source_iteration_ep}
	b(u^{k+\frac{1}{2}},v)=k(u^k,v) + \ell(v)\quad\text{for all } v\in \WW^+.
\end{align}
The half-step error $e^{k+\frac{1}{2}}=u-u^{k+\frac{1}{2}}$ satisfies
\begin{align}\label{eq:half-step_error}
	a(e^{k+\frac{1}{2}},v)=k(u^{k+\frac{1}{2}}-u^k,v)\quad\text{for all } v\in\WW^+.
\end{align}
The key idea is then to construct an easy-to-compute approximation to $e^{k+\frac{1}{2}}$ by Galerkin projection onto a suitable subspace $\WW^+_1\subset\WW^+$. This approximation is then used to correct $u^{k+\frac{1}{2}}$ to obtain a more accurate approximation $u^{k+1}$ to $u$.
Define the following closed subspace of $\WW^+$
\begin{align}\label{eq:subspace_def}
	\WW^+_1 = \{ u\in\WW^+: u=\P u\},
\end{align}
i.e., $\WW^+_1$ consists of functions in $\WW^+$ that do not depend on $\mu$.
The correction $u^{k+\frac{1}{2}}_D\in\WW_1^+$ is then computed by Galerkin projection of \eqref{eq:half-step_error} to $\WW_1^+$:
\begin{align}\label{eq:DSA_update}
	a(u_D^{k+\frac{1}{2}},v)=k(u^{k+\frac{1}{2}}-u^k,v)\quad\text{for all } v\in\WW^+_1,
\end{align}
and the new iterate is defined as
\begin{align}\label{eq:update_with_DSA}
	u^{k+1}=u^{k+\frac{1}{2}}+u_D^{k+\frac{1}{2}}.
\end{align}
If $u_D^{k+\frac{1}{2}}$ is a good approximation to $e^{k+\frac{1}{2}}$, then $e^{k+1}=e^{k+\frac{1}{2}}-u_D^{k+\frac{1}{2}}$ is small.
\matthias{The convergence proof of the iteration $\WW^+\to\WW^+$, $u^k\mapsto u^{k+1}$ is based on spectral analysis, see, e.g., \cite{Helmberg69,Werner00}.
\begin{lemma}\label{lem:convergence_half-step}
	The half-step error $e^{k+\frac{1}{2}}=u-u^{k+1/2}$ of the iteration \eqref{eq:source_iteration_ep} satisfies
	\begin{align*}
		\|e^{k+\frac{1}{2}}\|_a \leq c \|e^k\|_a,
	\end{align*}
	with constant $c=\|{\sigma_s}/{\sigma_t}\|_\infty$.
\end{lemma}
\begin{proof}
We endow $\WW^+$ with the inner product induced by the bilinear $b$ in this proof, and define bounded, self-adjoint and positive operators $A$ and $K$ on $\WW^+$ by
\begin{align*}
	b(Au,v)=a(u,v),\quad b(Ku,v)=k(u,v) \quad u,v\in\WW^+.
\end{align*}
Using $a=b-k$, we obtain that $A=I-K$. Using $u^{k+1/2}-u^k=e^k-e^{k+1/2}$, \eqref{eq:half-step_error} can be written as
\begin{align*}
	e^{k+1/2}= Ke^k.
\end{align*}
We thus have that
\begin{align*}
	\|e^{k+1/2}\|_a^2 = b((I-K)Ke^k,Ke^k)=b(K^2 (I-K)^{1/2}e^k, (I-K)^{1/2}e^k)&\leq \max\sigma(K)^2 \|(I-K)^{1/2}e^k\|_b^2 \\
	&\leq c^2\|e^k\|_a^2.
\end{align*}
In the last step we have used the following bounds on the numerical range
\begin{align*}
	0\leq b(Kv,v) = k(v,v) \leq \|\frac{\sigma_s}{\sigma_t}\|_\infty (\sigma_tv,v)\leq c b(v,v),
\end{align*}
which yields the spectral bounds $\sigma(K)\subset[0,c]$.
\end{proof}
}
\begin{theorem}\label{thm:conv_DSA_SI}
	Let Assumption (A1) hold, and let $c=\|{\sigma_s}/{\sigma_t}\|_\infty< 1$ be as in Lemma~\ref{lem:convergence_half-step}.  For any $u^0\in\WW^+$, the iteration defined by \eqref{eq:source_iteration_ep}, \eqref{eq:DSA_update}, \eqref{eq:update_with_DSA} converges linearly to the solution $u$ of Problem~\ref{prob:cont} with
	\begin{align*}
		\|u-u^{k+1}\|_a \leq \matthias{c}\|u-u^k\|_a.
	\end{align*}
\end{theorem}
\begin{proof}
	Since $u_D^{k+1}$ is the orthogonal projection of $e^{k+\frac{1}{2}}$ to $\WW^+_1$ in the $a$-inner product it holds that
	\begin{align*}
		\|e^{k+1}\|_a = \|e^{k+\frac{1}{2}}-u_D^{k+1}\|_a = \inf_{v\in\WW^+_1} \|e^{k+\frac{1}{2}}-v\|_a\leq \|e^{k+\frac{1}{2}}\|_a.
	\end{align*}
	The assertion then follows from Lemma~\ref{lem:convergence_half-step}.
\end{proof}
\begin{remark}
	The convergence analysis presented in this section carries over verbatim to multi-dimensional problems without symmetries, and it can be extended immediately to more general (symmetric and positive) scattering operators. \matthias{Moreover, if a Poincar\'e-Friedrichs inequality is available, cf., e.g., \cite{ManResSta00}, then the case $\sigma_a\geq 0$ can be treated similarly as long as $\sigma_t$ is uniformly bounded away from $0$, and the source iteration converges also in this situation.}
\end{remark}
\begin{remark}\label{rem:diffusion}
	Problem \eqref{eq:DSA_update} is the weak formulation of the diffusion equation
	\begin{align*}
		-\partial_z (\frac{1}{3\sigma_t} \partial_z u_D) + \sigma_a u_D &=  f\quad\text{in } (0,\slabwidth),
	\end{align*}
	with $f = \sigma_s \P(u^{k+\frac{1}{2}}-u^k)$, complemented by Robin boundary conditions, which shows the close relationship to DSA schemes.
\end{remark}
\begin{remark}\label{rem:alternative}
	The convergence analysis for the iteration without preconditioning, can alternatively be based on the following estimates. These estimates are the only ones in this paper that exploit that the scattering operator is related to the $L^2$-projector $\P$.
	First note that 
	\begin{align*}
		\|e^{k+\frac{1}{2}}\|_b^2=\|\P e^{k+\frac{1}{2}}\|_{\sigma_t}^2+\|(I-\P)e^{k+\frac{1}{2}}\|_{\sigma_t}^2+\|e^{k+\frac{1}{2}}\|_{L^2_-}^2+\|\mudz e^{k+\frac{1}{2}}\|_{\frac{1}{\sigma_t}}^2
	\end{align*}
	and that $\|e^{k+\frac{1}{2}}\|_b^2= k(e^k,e^{k+\frac{1}{2}})$. Setting $c=\|\sigma_s/\sigma_t\|_\infty$, the Cauchy-Schwarz inequality yields
	\begin{align*}
		(1-\frac{\varepsilon}{2})\|\P e^{k+\frac{1}{2}}\|_{\sigma_t}^2+\|(I-\P)e^{k+\frac{1}{2}}\|_{\sigma_t}^2+\|e^{k+\frac{1}{2}}\|_{L^2_-}^2+\|\mudz e^{k+\frac{1}{2}}\|_{\frac{1}{\sigma_t}}^2
		\leq \frac{c^2}{2\varepsilon}\|\P e^{k}\|_{\sigma_t}^2
	\end{align*}
	for any $\varepsilon \in (0,2]$. Choosing $\varepsilon=2$ shows that parts of the error are smoothened independently of $c$, i.e.,
	\begin{align*} \|(I-\P)e^{k+\frac{1}{2}}\|_{\sigma_t}^2+\|e^{k+\frac{1}{2}}\|_{L^2_-}^2+\|\mudz e^{k+\frac{1}{2}}\|_{\frac{1}{\sigma_t}}^2\leq \frac{c^2}{4}\|\P e^{k}\|_{\sigma_t}^2,
	\end{align*}
	while the angular average is hardly damped. In any case, this shows that $\P e^k$ converges to zero. It remains open how to exploit such an estimate to improve the analysis of the DSA preconditioned scheme above as it seems difficult to relate the latter smoothing property to the best approximation error of $e^{k+\frac{1}{2}}$ in the $a$-norm.
\end{remark}
\section{Galerkin approximations}\label{sec:galerkin}
In this section, we construct conforming approximation spaces $\WW_{h,N}^+\subset\WW^+$ in a two-step procedure. In a first step, we discretize the $\mu$-variable using discontinuous ansatz functions. In a second step, we discretize the $z$-variable by continuous finite elements.
Before stating the particular approximation space, we provide some general results. 
Let us begin with the definition of the discrete problem.
\begin{problem}\label{prob:disc}
 Let $q\in L^2(\DLmu)$, $g\in L^2_-$, $\WW_{h,N}^+\subset\WW^+$ and let $a$ and $\ell$ be defined as in Problem~\ref{prob:cont}. Find $u_h\in \WW_{h,N}^+$ such that for all $v\in\WW_{h,N}^+$ there holds
 \begin{align*}
 	a(u_h,v_h)=\ell(v_h).
 \end{align*}
\end{problem}
Since the bilinear form $a$ induces the energy norm that we use in our analysis, we immediately obtain the following best approximation result.
\begin{theorem}\label{thm:well_posed_disc}
	Let $\WW_{h,N}^+$ be a closed subspace of $\WW^+$. Then, there exists a unique solution $u_h\in\WW_{h,N}^+$ of Problem~\ref{prob:disc} that satisfies the a-priori estimate
	\begin{align}\label{eq:aprior_disc}
		\|u_h\|_a \leq  (\|q^+\|_{\frac{1}{\sigma_a}}^2+\|q^-\|_{\frac{1}{\sigma_t}}^2+2\|g_-\|_{L^2_-}^2)^{\frac{1}{2}},
	\end{align}
	and the following best approximation error estimate
	\begin{align}\label{eq:error_disc}
		\|u-u_h\|_{a}\leq \inf_{v_h\in\WW_{h,N}^+}\|u-v_h\|_{a}.
	\end{align}
\end{theorem}
In the next sections, we discuss some particular discretizations. We note that these generalize the spherical harmonics approach presented in \cite{EggerSchlottbom12}.
\subsection{$hp$ semidiscretization in $\mu$}\label{sec:semi_discrete}
Since we consider even functions, we require that the partition of the interval $[-1,1]$ for the $\mu$ variable respects the point symmetry $\mu\mapsto -\mu$. For simplicity, we thus partition the interval $[0,1]$ only, and the partition of $[-1,0]$ is implicitly defined by reflection.

Let $N\in\NN$ be a positive integer, and define intervals $\bar \mu_n=(\mu_{n-\frac{1}{2}},\mu_{n+\frac{1}{2}})$, $n=1,\ldots, N$, such that $\mu_{\frac{1}{2}}=0$ and $\mu_{N+\frac{1}{2}}=1$, and set $\Delta\mu_n=\mu_{n+\frac{1}{2}}-\mu_{n-\frac{1}{2}}$ and $\mu_n=(\mu_{n+\frac{1}{2}}+\mu_{n-\frac{1}{2}})/2$. Denote $\chi_n(\mu)$ the characteristic function of the interval $\bar\mu_n$. For $\mu>0$, we define the piecewise functions
\begin{align*}
	Q_{n,l}(\mu) = \sqrt{\frac{2l+1}{2}} P_l(2\frac{\mu-\mu_{n-\frac{1}{2}}}{\Delta\mu_n}-1) \chi_{\bar\mu_n}(\mu), \quad \mu >0,
\end{align*}
where $P_l$ denotes the $l$th Legendre polynomial. Hence, $\{Q_{n,l}\}_{l=0}^L$ is an $L^2$-orthonormal basis for the space of polynomials of degree $L$ on each interval $\bar\mu_n$.
For $\mu>0$, we set $Q_{n,l}^\pm(\mu)=Q_{n,l}(\mu)$, and for $\mu<0$, we set $Q_{n,l}^\pm(\mu)=\pm Q_{n,l}^\pm(-\mu)$.
The semidiscretization of the even parts is then
\begin{align*}
	u(z,\mu) \approx u_h(z,\mu)=\sum_{n=1}^N\sum_{l=0}^{L} \phi_{n,l}^+(z) Q_{n,l}^+(\mu).
\end{align*}
\begin{remark}
	If we partition the interval $[-1,1]$ for the angular variable by a single element, we obtain truncated spherical harmonics approximations, see, e.g., \cite{LewisMiller84,EggerSchlottbom12}. Partitioning of $[-1,1]$ into two symmetric intervals $(-1,0)\cup(0,1)$ corresponds to the double $P_L$ method \cite{LewisMiller84}, which generalizes in multiple dimensions to half space moment methods \cite{Dubroca03}.
	The latter can resolve the non-smoothness of $\phi$ at $\mu=0$, and, thus might yield spectral convergence on the intervals $\mu>0$ and $\mu<0$.
\end{remark}
\subsection{Fully discrete scheme}
In order to obtain a conforming discretization, we approximate the coefficient functions $\phi_{n,l}^+$ using $H^1(0,\slabwidth)$-conforming elements.
Let $J\in\NN$, and $\bar z_j = (z_{j-1},z_{j})$ such that
\begin{align*}
	[0,\slabwidth]=\cup_{j=1}^{J} {\rm clos}(\bar z_j)
\end{align*}
be a partition of $(0,\slabwidth)$. Let $p\geq 1$ and denote $P_p$ the space of polynomials of degree $p$.
The full approximation space is then defined by
\begin{align}\label{eq:w1}
	\WW_{h,N}^+ &=\{ \psi_h^+(z,\mu)=\sum_{n=1}^N\sum_{l=0}^L \psi_{n,l}^+(z) Q_{n,l}^+(\mu): \psi_{n,l}^+(z)\in H^1(0,\slabwidth),\ \psi^+_{n,l\mid \bar z_j}\in P_p\}.
\end{align}
The choice \eqref{eq:w1} for the approximation space $\WW_{h,N}^+$ corresponds to an {\it hp} finite element method, for which the assertion of Theorem~\ref{thm:well_posed_disc} holds true.

\begin{remark}\label{rem:connection_to_mixed}
	If the solution $u_h\in\WW_{h,N}^+$ to Problem~\ref{prob:disc} is computed, we compute even and odd approximations to the solution $\phi$ of \eqref{eq:rte1} via $\phi_h^+=u_h$ and the odd part $\phi_h^-\in\WW_h^-$ as the solution to the variational problem $(\phi_h^-,\psi_h^-)=(\frac{1}{\sigma_t}(q^--\mudz u_h),\psi_h^-)$ for all $\psi_h^-\in\WW_h^-$, where 
	\begin{align*}
			\WW_{h,N}^- &=\{ \psi_h^-(z,\mu)=\sum_{n=1}^N\sum_{l=0}^{L+1} \psi_{n,l}^-(z) Q_{n,l}^-(\mu): \psi^-_{n,l\mid \bar z_j}\in P_{p-1}\}.
	\end{align*}
	We note that this space satisfies the compatibility condition $\mudz \WW_{h,N}^+\subset\WW_{h,N}^-$, which makes this pair of approximation spaces suitable for a direct approximation of a corresponding mixed formulation, cf. \cite{EggerSchlottbom12}. The even-parity formulation that we consider here corresponds then to the Schur complement of the mixed problem, cf. \cite{EggerSchlottbom12}. The reader should note the different degrees in the polynomial approximations, e.g., if $\phi_h^+$ is piecewise constant in angle, then $\phi_h^-$ is piecewise linear.
\end{remark}
\section{Discrete preconditioned source iteration}\label{sec:sn}

In order to solve the discrete variational problem defined in Problem~\ref{prob:disc}, we proceed as in Section~\ref{sec:dsa_source} but with $\WW^+$ and $\WW^+_1$ replaced by $\WW^+_{h,N}$ and $\WW^+_{h,1}$, respectively. We note that $\WW^+_{h,1}\subset \WW^+_1$ consists of functions in $\WW^+_{h,N}$ that do not dependent on $\mu$.

The finite dimensional counterpart of the DSA preconditioned source iteration is then defined as follows:
For given $u_h^k\in\WW_{h,N}^+$, compute $u_h^{k+\frac{1}{2}}\in\WW_{h,N}^+$ as the unique solution to
\begin{align}\label{eq:source_iteration_ep_disc}
	b(u_h^{k+\frac{1}{2}},v_h)=k(u^k_h,v_h) + \ell(v_h)\quad\text{for all } v_h\in \WW^+_{h,N}.
\end{align}
The correction $u^{k+\frac{1}{2}}_{h,D}\in\WW_{h,1}^+$ is defined by Galerkin projection of $e_h^{k+\frac{1}{2}}$ to $\WW_{h,1}^+$:
\begin{align}\label{eq:DSA_update_disc}
	a(u_{h,D}^{k+\frac{1}{2}},v_h)=k(u_h^{k+\frac{1}{2}}-u_h^k,v_h)\quad\text{for all } v_h\in\WW^+_{h,1},
\end{align}
and the new iterate is defined as
\begin{align}\label{eq:update_with_DSA_disc}
	u^{k+1}_h=u^{k+\frac{1}{2}}_h+u_{h,D}^{k+\frac{1}{2}}.
\end{align}
Using the same arguments as above, we obtain the following convergence result.
\begin{theorem}\label{thm:conv_DSA_SI_disc}
	Let Assumption (A1) hold, and let $c<1$ be as in Lemma~\ref{lem:convergence_half-step}.  For any $u_h^0\in\WW^+_{h,N}$, the iteration defined by \eqref{eq:source_iteration_ep_disc}, \eqref{eq:DSA_update_disc}, \eqref{eq:update_with_DSA_disc} converges linearly with
	\begin{align*}
		\|u_h-u_h^{k+1}\|_a \leq \matthias{c}\|u_h-u_h^k\|_a.
	\end{align*}
\end{theorem}

\begin{remark}\label{rem:diffusion_disc}
	Similar to Remark~\ref{rem:diffusion}, $u_{h,D}^{k+\frac{1}{2}}$ is the Galerkin projection to $\WW_{h,1}^+$ of the solution to 
	\begin{align*}
		-\partial_z (D(z) \partial_z u_D) + \sigma_a u_D&= f,\quad 0<z<\slabwidth,
	\end{align*}
	with $\mu$-grid dependent diffusion coefficient $D(z)$ and $f=\sigma_s\P(u_h^{k+\frac{1}{2}}-u_h^k)$. If piecewise constant functions in angle are employed, then $D(z)=\frac{1}{3\sigma_t(z)}(1+ \frac{1}{4}\sum_{n=1}^N \Delta\mu_n^3)$. 
\end{remark}

\begin{remark}
	Once the scattering term in the right-hand side of \eqref{eq:source_iteration_ep_disc} has been computed, the half-step iterate $u_h^{k+\frac{1}{2}}$ can be computed independently for each direction, and, thus, its computation can be parallelized.
\end{remark}

\section{Numerical examples}\label{sec:numerics}
In this section, we report on the accuracy of the proposed discretization scheme and its efficient numerical solution using the DSA preconditioned source iteration of Section~\ref{sec:sn}. 
We restrict our discussion to a low-order method that offers local resolution. To do so, we fix $p=1$ and $L=0$ in the definition of $\WW^+_{h,N}$, while $N$ might be large. Hence, the approximation space $\WW^+_{h,N}$ consists of discontinuous, piecewise constant functions in the angular variable and continuous, piecewise linear functions in $z$.

\subsection{Manufactured solutions}
To investigate the convergence behavior, we use manufactured solutions, i.e., the exact solution is
\begin{align}\label{eq:manu1}
	\phi(z,\mu)=|\mu|e^{-\mu}e^{-z(1-z)},
\end{align}
with parameters $\sigma_a=1/100$, $\sigma_s(z)=2+\sin(\pi z)/2$ and $\slabwidth=1$, and source terms defined accordingly.
We computed the numerical solution $u_h$ using the DSA preconditioned iteration \eqref{eq:source_iteration_ep_disc}, \eqref{eq:DSA_update_disc}, \eqref{eq:update_with_DSA_disc}.
We stopped the iteration using the a-posteriori stopping rule
\begin{align*}
	\|u_h^{k}-u_h^{k-1}\|_{a}\leq \varepsilon,
\end{align*}
where $\varepsilon=10^{-10}$ is a chosen tolerance.
The approximations of the even and odd parts of the solutions are recovered as described in Remark~\ref{rem:connection_to_mixed}.

For the chosen parameters, we have that $\kappa\leq 250$, which leads to a predicted convergence rate of $0.996$. 
Table~\ref{tab:manu1} shows the errors for different discretization parameters $N$ and $J$. As expected we observe a linear rate of convergence with respect to $N$, cf. Theorem~\ref{thm:well_posed_disc}.
In addition, we observe that the preconditioned source iteration converged within at most $15$ iterations, and the expression $\|u_h^{k}-u_h^{k-1}\|_{a}$ decreased by $0.21$ in each iteration.
The convergence rate with respect to $J$ is initially quadratic, which is better than predicted by Theorem~\ref{thm:well_posed_disc}, and then saturates for fixed $N=8192$. As before, the preconditioned source iteration converged within at most $15$ iterations with a decrease by a factor of $0.21$ of $\|u_h^{k}-u_h^{k-1}\|_{a}$.

The observed convergence rate is thus as the one obtained by classical Fourier analysis for constant coefficients and periodic boundary conditions, which is bounded by $0.2247$ \cite{AdamsLarsen02}. 
Furthermore, we observe that the rate of convergence does not dependent on the grid size as predicted by Theorem~\ref{thm:conv_DSA_SI_disc}, cf. Remark~\ref{rem:diffusion_disc}, i.e., using the terminology of \cite{AdamsLarsen02}, our discrete diffusion approximation is consistently discretized.


\begin{table}[ht!]
\centering\small\setlength\tabcolsep{1em}
\begin{tabular}{r c c} 
	\toprule
$N$ & $E_h$ & rate\\
\midrule
 512 & 1.61e-04 &    \\ 
1024 & 8.07e-05 & 0.99 \\
2048 & 4.04e-05 & 0.99 \\
4096 & 2.04e-05 & 0.99 \\
8192 & 1.06e-05 & 0.95 \\
\bottomrule
\end{tabular}\quad \quad
\begin{tabular}{r c c} 
	\toprule
$J $ & $E_h$ & rate\\
\midrule
16&7.88e-04 & \\
32&1.99e-04 &1.98 \\
64&5.14e-05 &1.96 \\
128&1.63e-05 &1.66 \\
256&1.06e-05 &0.62 \\
\bottomrule
\end{tabular}
\caption{Observed errors $E_h=\|\phi-\phi_h\|_{L^2(\DLmu)}$ between finite element solution $\phi_h$ and the manufactured solution $\phi$ defined in \eqref{eq:manu1} together with the rate of convergence of $E_h$. Left: Convergence for different discretization parameters $N$, and $J=256$. Right: Convergence for different discretization parameters $J$ and $N=8192$. 
\label{tab:manu1}}
\end{table}

In the next section, we present a more detailed study of spectrum of the preconditioned operator.

\subsection{Eigenvalue studies of the preconditioned operator}
In this section, we consider the spectrum of the error propagation operator $\P e_h^{n}\mapsto \P e_h^{n+1}$, where $\P$ is the $L^2$-projection onto constants in angle defined in Section~\ref{sec:prelim} and $e_h^{n}$ is the sequence of errors generated by the DSA preconditioned source iteration, cf. Remark~\ref{rem:alternative}.
The function $\P e_h^{n}$ depends only on $z$, and, assuming $\sigma_s>0$, we can measure the projected error using the norm induced by $\sigma_s$, i.e.,
\begin{align*}
	\|\P e_h^n\|_{\sigma_s} = \|e_h^n\|_k.
\end{align*}
We choose the following scattering and absorption parameters
\begin{align*}
	\sigma_s(z)=\begin{cases}2+\sin(2\pi z), & z\leq \frac{1}{2}\\ 102+\sin(2\pi z), & z>\frac{1}{2},\end{cases}\qquad
	\sigma_a(z)=\begin{cases}10.01, & z\leq \frac{1}{2}\\ 0.01, & z>\frac{1}{2}.\end{cases}
\end{align*}
We notice that both parameters have huge jumps and that \matthias{the predicted convergence rate is $c=\|\sigma_s/\sigma_t\|_\infty \approx 0.9999$, cf. Theorem~\ref{thm:conv_DSA_SI}}.

In Figure~\ref{fig:eigenvalues} we plot the spectrum of the error propagation operator for different mesh sizes. 
All eigenvalues are bounded from above by $0.2247$, which is inline with the results of \cite{MarchukLebedev86,AdamsLarsen02}. For each spatial discretization, we observe that the corresponding eigenvalues are monotonically increasing with $N$. The spectra for $N\geq 16$ lie on top of each other, indicating convergence of the eigenvalues. In particular, also for the coarse grid approximations the spectrum is uniformly bounded by $0.2247$, again confirming the robustness of the method with respect to different approximations.

\begin{figure}
	\includegraphics[width=0.3\textwidth]{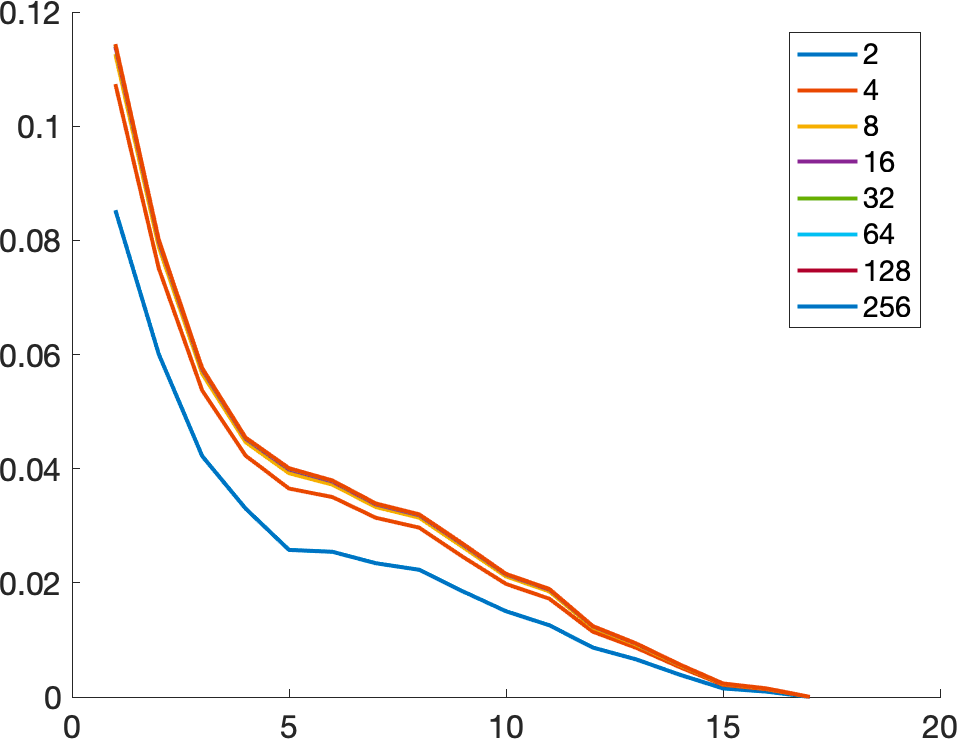}\hfill
	\includegraphics[width=0.3\textwidth]{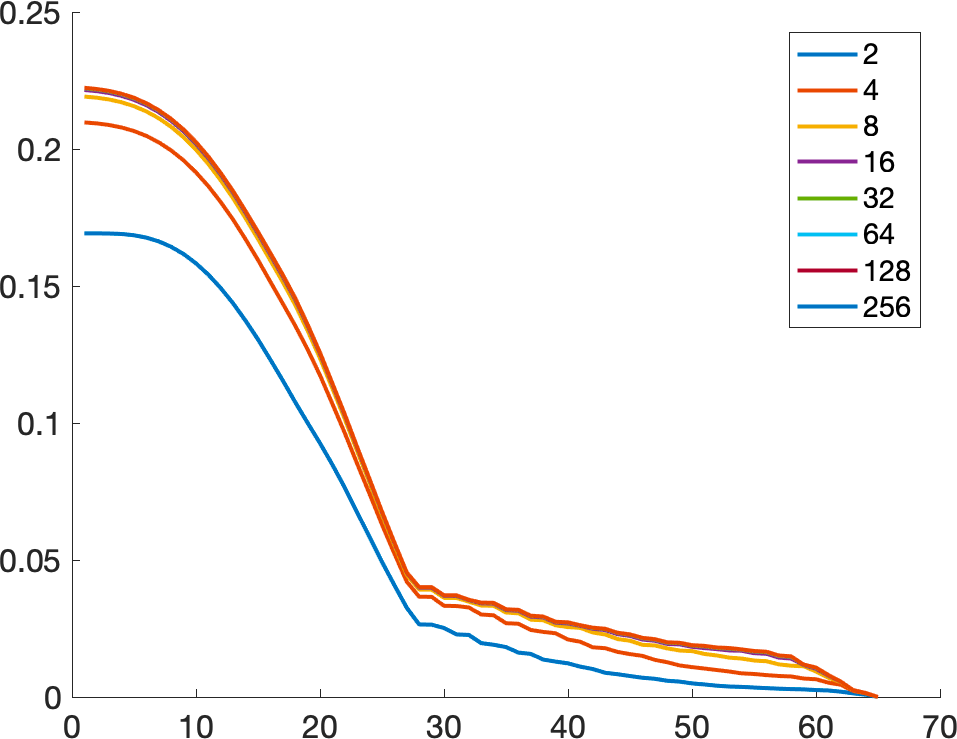}\hfill
	\includegraphics[width=0.3\textwidth]{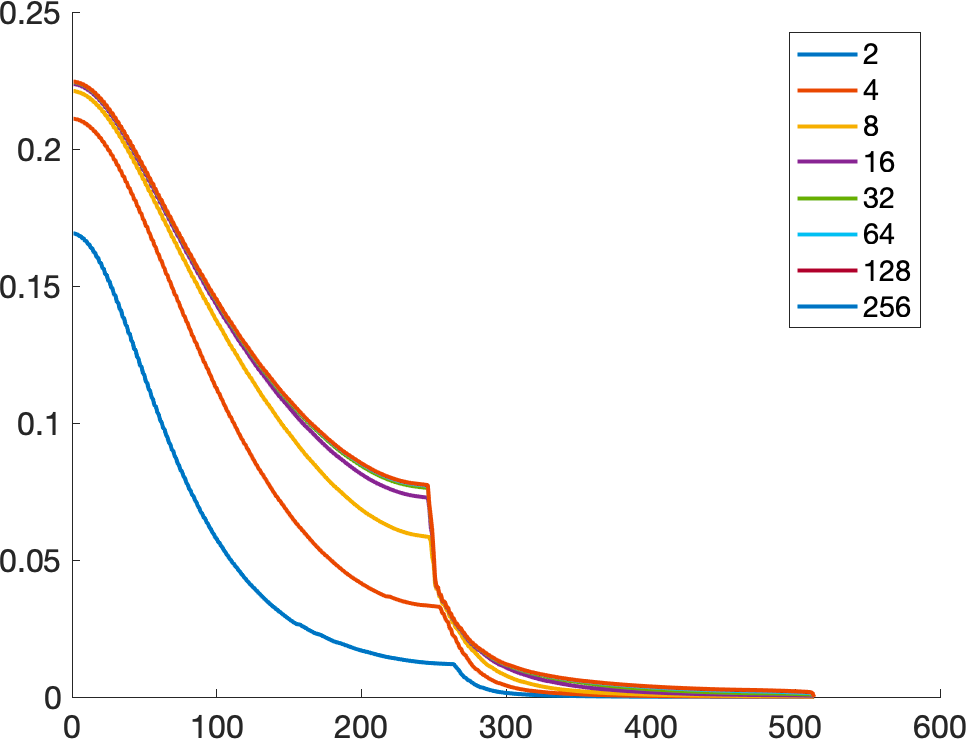}		
	\caption{\label{fig:eigenvalues} Spectra of the error propagation operator $\P e_h^{n}\mapsto \P e_h^{n+1}$ for different spatial discretizations $J=16,64,512$ (from left to right). Each plot contains the corresponding spectra for $N=2^i$, $i=1,\ldots,8$.}
\end{figure}

\subsection{Multi-D: The lattice problem}
Although the theory has been presented for slab problems, our results carry over verbatim to other problems.
In the following we report on the performance of the method for the lattice problem, cf. \cite{Brunner05}.
This test problem is a three-dimensional problem with certain symmetries. The radiative transfer equation here writes as
\begin{align*}
	s \cdot\nabla \phi(x,s) + \sigma_t(x) \phi(x,s)=\sigma_s(x)\P\phi + q(x,s)
\end{align*}
where $x\in\R\subset\RR^2$ and $s\in\mathbb{S}^2\subset\RR^3$. 
The computational domain is a square $\R=(0,7)\times(0,7)$, and homogeneous inflow boundary conditions are considered. The absorption and scattering rates are piecewise constant functions. We define $\sigma_a=10$ in the black regions shown in Figure~\ref{fig:geometry}, and $\sigma_a=0$ elsewhere. 
We set $\sigma_s=1$ in the grey and white regions and $\sigma_s=0$ elsewhere. The source is defined by $q(x,s)=1$ in the white region and $q(x,s)=0$ elsewhere. Note that due to the availability of a Poincar\'e-Friedrichs inequality \cite{ManResSta00}, the case $\sigma_a=0$ leads to a well-posed radiative transfer problem, and, since $\sigma_s+\sigma_a\geq 1$, the theory presented here is applicable. Moreover, the constant \matthias{$c$} will depend on the constant from the Poincar\'e-Friedrichs inequality and \matthias{$c<1$} even if $\sigma_a$ vanishes.

We discretize $L^2(\mathbb{S}^2)$ by approximating $\mathbb{S}^2$ using a geodesic polyhedron consisting of flat triangles, see Figure~\ref{fig:geometry}. The approximation space in angle then consists of standard discontinuous finite element spaces associated to this triangulation. In the following, we focus on piecewise constant approximations, but higher order approximations are straight-forward if the geometry approximation is also of higher order.

In Figure~\ref{fig:two_solutions} we show the angular averages of the computed solutions for two different grids
 with $J=9\,801$ vertices in the spatial grid and $N=4$ triangles on a half-sphere and for $J=78\,961$ and $N=64$, respectively.
We note that our solutions do not exhibit ray effects, cf. \cite{LewisMiller84,Brunner05}. 
%
The preconditioned source iteration converged with $9$ and $17$ iterations for the coarse grid approximation and for the fine grid approximation, respectively. This amounts to an error reduction per iteration of at least $0.04$ for the coarse grid discretization and of at least $0.20$ for the fine grid discretization.

\begin{figure}
	\includegraphics[width=0.3\textwidth]{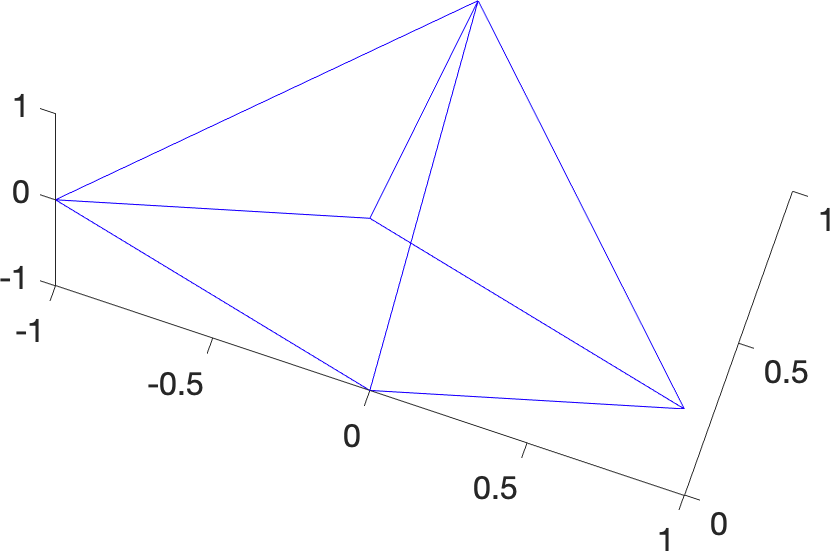}\hfill
	\includegraphics[width=0.3\textwidth]{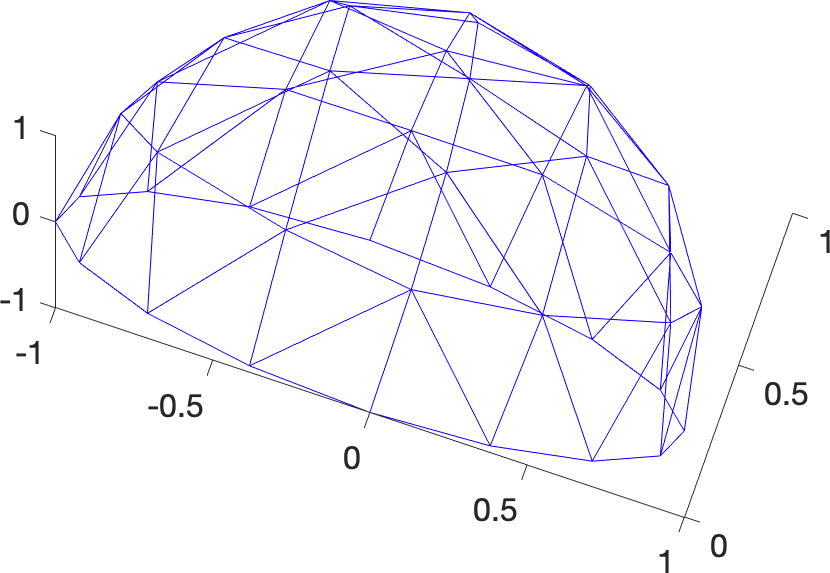}\hfill
	\includegraphics[width=0.3\textwidth]{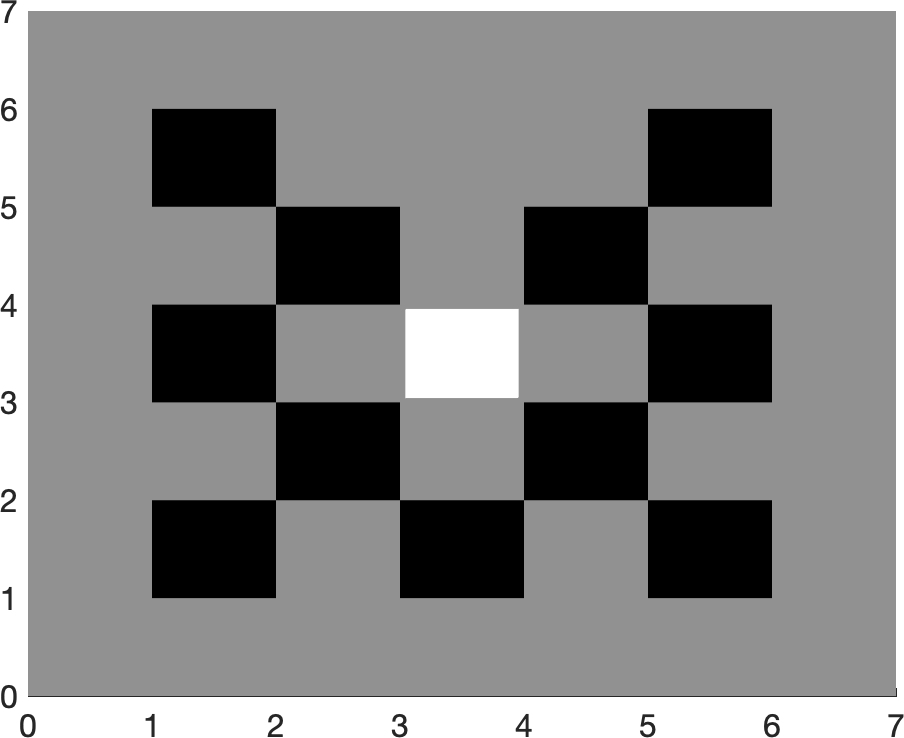}
	\caption{Left and middle: Approximation of the half-sphere with $N=4$ and $N=64$ triangles. Right: Geometry of the lattice problem. \label{fig:geometry}}
\end{figure}

\begin{figure}
	\includegraphics[width=.49\textwidth]{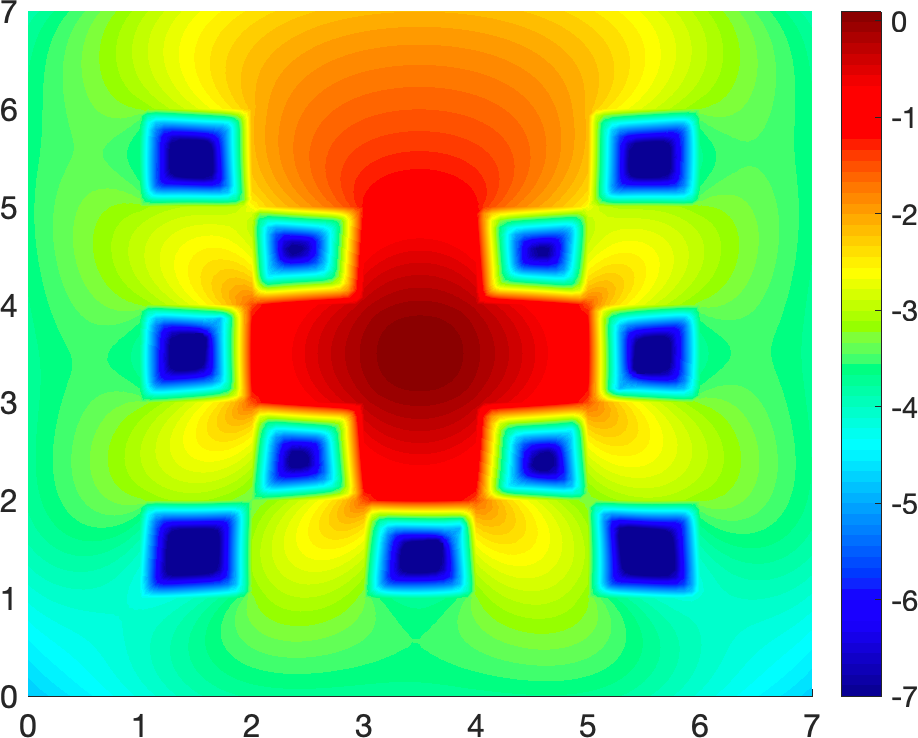}\hfill
	\includegraphics[width=.49\textwidth]{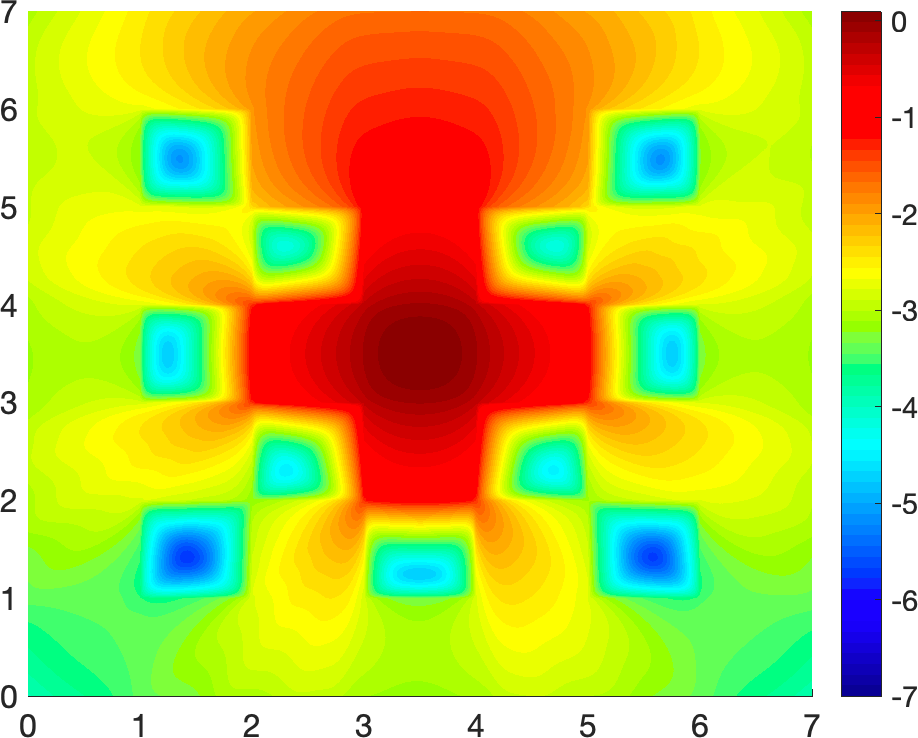}
	\caption{Angular average of the computed solution in a $\log_{10}$-scale for the lattice problem for $J=9\,801$ spatial vertices and $N=4$ triangles on a half-sphere (left) and $J=78\,961$ spatial vertices and $N=64$ triangles on a half-sphere (right).\label{fig:two_solutions}}
\end{figure}

Next, let us investigate the behavior of the preconditioned source iteration for scaled parameters in more detail. Introducing a scale parameter $\delta>0$, a diffusion limit is obtained by the scaling \cite{DautrayLions6}
\begin{align*}
	\frac{\bar \sigma_s(x)}{\delta},\quad \delta \bar \sigma_a(x), \quad \delta \bar q(x),
\end{align*}
where both parameters $\bar \sigma_s$ and $\bar \sigma_a$ are bounded and strictly bounded away from zero. Since this is not the case for the lattice problem, we consider the following parameters
\begin{align*}
	\sigma_s^\delta(x) =\frac{\sigma_s(x)+1/10}{\delta},\qquad \sigma_a^\delta=\delta(\sigma_a(x)+1/10),\qquad q^\delta(x,s)=\delta q(x).
\end{align*}
For $\delta\to 0$, the corresponding solution $u^\delta$ will converge to the solution of a diffusion problem; for non-smooth coefficients see \cite{EggerSchlottbom2014}. The parameter \matthias{$c$ defined in Lemma~\ref{lem:convergence_half-step} is bounded by $O(1/\delta)$}.
\begin{table}[ht!]
\centering\small\setlength\tabcolsep{1em}
\begin{tabular}{r c c c c c c c c c} 
	\toprule
& \multicolumn{4}{c}{$J=9\,801$} && \multicolumn{4}{c}{$J=78\,961$}\\ &\multicolumn{2}{c}{$N=4$}&\multicolumn{2}{c}{$N=64$}&&\multicolumn{2}{c}{$N=4$}&\multicolumn{2}{c}{$N=64$}\\
	 \cmidrule{2-5}	  \cmidrule{7-10}
$1/\delta$ & $k$ & rate & $k$ & rate && $k$ & rate & $k$ & rate\\
\midrule
 1  & 9 & 0.04  & 15 & 0.16 &&  9& 0.04 & 15 & 0.17\\ 
10  & 9 & 0.06  & 15 & 0.22 &&  9& 0.06 & 16 & 0.25\\
100 & 8 & 0.06  & 13 & 0.22 &&  9& 0.07 & 15 & 0.27\\
1000& 5 & 0.01  &  7 & 0.06  &&  6& 0.05 & 10 & 0.17\\
\bottomrule
\end{tabular}
\caption{Iteration counts $k$ and minimal reduction rates for $\|u_h^k-u_h^{k-1}\|_a$ for the lattice problem with scaled parameters $\sigma_s^\delta$, $\sigma_a^\delta$ and $q^\delta$ for different $\delta$ and discretizations with $N$ triangles on a half-sphere and $J$ vertices in the spatial mesh.
\label{tab:diffusion_scaling}}
\end{table}
Table~\ref{tab:diffusion_scaling} shows the iteration counts and the minimum reduction of $\|u_h^k-u_h^{k-1}\|_a$ during the iteration. We observe that the preconditioned iteration is robust with respect to $\delta\to 0$ for different meshes.
The convergence rate on the finest grid is, however, slightly worse than the convergence rate for slab problems.

\section{Conclusions}\label{sec:conclusion}

We investigated discontinuous angular and continuous spatial approximations of the even-parity formulation for the radiative transfer equation.
Certain instances of these approximations are closely related to classical discretizations, such as truncated spherical harmonics approximations, double $P_L$-methods or discrete ordinates methods.

We considered a diffusion accelerated preconditioned source iteration for the solution of the resulting variational problems that has been formulated in infinite dimensions. Convergence rates of this iteration have been proven. 
The Galerkin approach used for the discretization allowed to translate the results for the infinite dimensional iteration directly to the discrete problems. Moreover, the discrete iteration converges independently of the chosen resolution.

The theoretically proven convergence rate in Theorem~\ref{thm:conv_DSA_SI} is not robust in the limit of large scattering, while numerical results show that in practice the preconditioned iteration converges robustly even in scattering dominated problems. 
One approach to obtain an improved convergence rates estimate is to estimate the best approximation error $\inf_{v\in\WW^+_1} \|e^{k+\frac{1}{2}}-v\|_a$, which, however, seems rather difficult, and we postpone a corresponding rigorous analysis to future research.

The term diffusion acceleration is linked to the usage of the space $\WW_1^+$ in \eqref{eq:subspace_def} that consists of functions constant in angle. This choice is, however, not essential, and other closed subspace can be employed, which allows for the construction of multi-level schemes.
The multi-level approach might lead to a feasible approach to estimate the best approximation error.

In any case, the DSA preconditioner can be combined with a conjugate gradient method in order to further reduce the number of iterations. Moreover, unlike many discrete ordinates schemes, the  numerical approximation did not suffer from the ray effect in our numerical examples.

\section*{Acknowledgments}
The authors would like to thank Herbert Egger and Markus Bachmayr for stimulating discussions.

\bibliographystyle{elsarticle-num}
\bibliography{mixed_sn_slab}
\end{document}